\documentclass[12pt]{amsart}
\usepackage[active]{srcltx}
\usepackage{a4wide}
\usepackage{amssymb,amsbsy,amsmath,amsfonts,amssymb,amscd, mathrsfs}
\usepackage{dsfont}
\usepackage[T1]{fontenc}
\usepackage[utf8]{inputenc}
\usepackage{fourier}
\usepackage[dvipsnames,usenames]{color}

\newcommand{\N}{\mathbb{N}}

\newcommand{\BMO}{\operatorname{BMO}}
\newcommand{\BMOA}{\operatorname{BMOA}}

\newcommand{\RR}{\mathbb{R}}
\newcommand{\C}{\mathbb{C}}
\newcommand{\Z}{\mathbb{Z}}
\newcommand{\T}{\mathbb{T}}

\newcommand{\Hi}{{\mathscr{H}}^\infty}

\newcommand{\Real}{\operatorname{Re}}

\newtheorem{theorem}{Theorem}

\newtheorem{lemma}[theorem]{Lemma}
\newtheorem{corollary}[theorem]{Corollary}
\usepackage{color}
\newcommand{\fin}{{ fin}}

\numberwithin{equation}{section}
\numberwithin{theorem}{section}

\begin{document}

\title[
Riesz projection and bounded mean oscillation for Dirichlet series]{
Riesz projection and bounded mean oscillation for Dirichlet series}

\author[S. Konyagin]{Sergei Konyagin}
\address{Sergei Konyagin, Steklov Institute of Mathematics,
8 Gubkin Street, Moscow,
119991, Russia}
\email{konyagin@mi-ras.ru}

\author[H. Queff\'{e}lec]{Herv\'{e} Queff\'{e}lec}

\address{Herv\'{e} Queff\'{e}lec, Universit\'{e} Lille Nord de France, USTL, Laboratoire Paul Painlev\'{e} UMR. CNRS 8524, F--59 655 Villeneuve d'Ascq
Cedex, France}\email{herve.queffelec@univ-lille.fr}
\author[E.~Saksman]{Eero Saksman}
\address{Eero Saksman \\
Department of Mathematics and Statistics \\
University of Helsinki \\
FI-00170 Helsinki \\
Finland} \email{eero.saksman@helsinki.fi}

\author[K.~Seip]{Kristian Seip}
\address{Kristian Seip, Department of Mathematical Sciences, Norwegian University of Science and Technology,
NO-7491 Trondheim, Norway} \email{kristian.seip@ntnu.no}

\thanks{Konyagin was supported from a grant to the Steklov International Mathematical
Center in the framework of the national project ''Science'' of the Russian Federation,
 Queff\'{e}lec was supported in part by the Labex CEMPI (ANR-11-LABX-0007-01), 
Saksman was supported in part by the Finnish Academy grant 1309940, and Seip was supported in part by the Research Council of Norway grant 275113.} 



\begin{abstract}
We prove that the norm of the Riesz projection from $L^\infty(\Bbb{T}^n)$ to $L^p(\Bbb{T}^n)$ is 
$1$ for all $n\ge 1$ only if $p\le 2$, thus solving a problem posed by Marzo and Seip in 2011. 
This shows that $H^p(\Bbb{T}^{\infty})$ does not contain the dual space of 
$H^1(\Bbb{T}^{\infty})$ for any $p>2$. We then note that the dual of $H^1(\Bbb{T}^{\infty})$ contains, via the Bohr lift, the space of Dirichlet series in $\operatorname{BMOA}$ of the right half-plane. We give several conditions showing how this $\BMOA$ space relates to other spaces of Dirichlet series. 
Finally, relating the partial sum operator for Dirichlet series to Riesz projection on $\T$, we compute its $L^p$ norm when $1<p<\infty$, and we use this result to show that the $L^\infty$ norm of the $N$th partial sum of a bounded Dirichlet series over $d$-smooth numbers is of order $\log\log N$.
\end{abstract}

\maketitle


\section{Introduction}
This paper is concerned with two different ways of transferring Riesz projection to the infinite-dimensional setting of Dirichlet series: first, by lifting it in a multiplicative way to the infinite-dimensional torus $\T^{\infty}$ and second, by using one-dimensional Riesz projection to study the partial sum operator acting on Dirichlet series. In either case, we will be interested in studying the action of the operator in question on functions in $L^p$ or $H^p$ spaces. 

By Fefferman's duality theorem \cite{F},  Riesz projection $P_{1}^{+}$ on the unit circle $\T$, formally defined as 
\[
P_{1}^{+}\big(\sum_{k\in \Z}c_k z^k\big):=\sum_{k\geq 0}c_k z^k ,
\]
maps $L^{\infty}(\T)$ into and onto $\BMOA(\T)$, i.e., the space of analytic functions of bounded mean oscillation. We may thus think of the image of $L^{\infty}(\T^{\infty})$ under Riesz projection on $\T^{\infty}$
(or equivalently, in view of the Hahn--Banach theorem, the dual space $H^1(\T^\infty)^*$) as a possible infinite-dimensional counterpart to $\BMOA(\T)$. 
This brings us to the second main topic of this paper which is to describe some of the main properties of this space.  

Our main result, given in Section~\ref{sec:Riesz}, verifies that Riesz projection does not map 
$L^{\infty}(\T^{\infty})$ into $H^p(\T^{\infty})$ for any $p>2$, whence $H^1(\T^\infty)^*$ is not embedded in $H^p(\T^{\infty})$ for any $p>2$. This result solves a problem posed in \cite{MS} and contrasts the familiar inclusion of  $\BMOA(\T)$ in $H^p(\T)$ for every $p<\infty$. 
The key idea of the proof is to first show that the norm of a Fourier multiplier 
$M_{\chi_A}:L^p(\T^n)\to L^q(\T^n)$ corresponding to a bounded convex domain $A$ 
in $\RR^n$ is dominated by the norm of the Riesz projection on $\T^{n+m}$ for $m$ sufficiently large, depending on $A$. Another crucial ingredient is Babenko's well-known lower estimate for spherical Lebesgue constants.

We then proceed to view $H^1(\T^\infty)^*$ as a space of Dirichlet series, employing as usual the Bohr lift. This leads us in Section \ref{sec:BMOA} to a distinguished subspace of $H^1(\T^\infty)^*$ which is indeed a ``true'' $\BMO$ space, namely the family of Dirichlet series that belong to $\BMOA$ of the right half-plane. By analogy with classical results on $\T$, we give several conditions for membership in this space, also for randomized Dirichlet series, and we describe how this $\BMOA$ space relates to some other function spaces of Dirichlet series.

In Section \ref{sec:compare}, we study Dirichlet polynomials of fixed length $N$ and compare the size of their norms in $H^p$, $\BMOA$, and the Bloch space.  One of these results is then applied in the final Section \ref{sec:partial}, where we turn to our second usage of Riesz projection. Here we present an explicit device for expressing the $N$th partial sum of a Dirichlet series in terms of one-dimensional Riesz projection and give some $L^p$ estimates for the associated partial sum operator. 

We refer the reader to \cite{HLS} and \cite{MAHE} (see especially \cite[Section 6]{MAHE}) for definitions and basics on Hardy spaces of Dirichlet series of Hardy spaces on $\T^\infty$.
\subsection*{Notation} We will use the notation $f(x)\ll g(x)$ if there is some constant $C>0$ such that $|f(x)|\le C|g(x)|$ for all (appropriate) $x$. If we have both $f(x)\ll g(x)$ and $g(x)\ll f(x)$, then we will write $f(x)\asymp g(x)$. If $\lim_{x\to \infty} f(x)/g(x)=1$, then we write $f(x)\sim g(x)$. 

\subsection*{Acknowledgements} We thank Ole Fredrik Brevig for allowing us to include an unpublished argument of his in this paper. We are also grateful to the referees for a number of valuable comments that helped improve the presentation.

\section{The norm of the Riesz projection from $L^\infty(\Bbb{T}^n)$ to $L^p(\Bbb{T}^n)$}\label{sec:Riesz}

The norm $\|f\|_p$ of a function $f$ in $L^p(\T^\infty)$ is computed with respect to 
Haar measure $m_\infty$ on $\T^\infty$, which is the countable product of one-dimensional normalized Lebesgue measures on $\T$.  We denote by $m_n$ the measure on $\T^n$ that is  the  $n$-fold product of the normalised one-dimensional measures, and $L^p(\T^n)$ is defined with respect to this measure.

We write the Fourier series of a function $f$ in $L^1(\T^n)$ on the $n$-torus $\T^n$ as
\begin{equation}
\label{def_expan}
f(\zeta) = \sum_{\alpha\in\Z^n} \hat f(\alpha) \zeta^{\alpha}.
\end{equation}
For a function $f$ in $L^1(\T^\infty)$ the Fourier series takes the form 
$f(\zeta) = \sum_{\alpha\in\Z^\infty_\fin} \hat f(\alpha) \zeta^{\alpha}$, where 
$\Z^\infty_\fin$ stands for infinite multi-indices such that all but finitely many 
indices are zero. We also set $\Z_+:=\{0,1,\dots\}$ so that $\Z_+^n$ (respectively  $\Z_+^\infty$)
is the positive cone in  $\Z^n$  (respectively  $\Z^\infty$).
The operator
\[ P_{n}^+ f(\zeta) := \sum_{\alpha\in\Z_+^n} \hat f(\alpha) \zeta^{\alpha} \]
is the Riesz projection on $\T^n$, and, as an operator on $L^2(\T^n)$, it has norm $1$.
If we instead view $P_{n}^+$ as an operator on $L^p(\T^n)$ for $1<p<\infty$, then 
a theorem of Hollenbeck and Verbitsky \cite{HV} asserts that its norm equals
$(\sin(\pi/p))^{-n}$. In an analogous way we denote by $P^+_\infty$ the Riesz projection on $\T^\infty,$ and  obviously  $P^+_\infty$ is bounded on $L^p(\T^\infty)$ only for $p=2$, when its norm equals 1.

Using this normalization, we let $\|P_n^+\|_{q,p}$ denote the norm of 
the operator $P_n^+:\,L^q(\T^n)\to L^p(\T^n)$ for $q\ge p$. By H\"older's 
inequality, $p\to \|P_n^+\|_{\infty,p}$ is a continuous and nondecreasing function, 
and obviously
$\|P^+_n\|_{\infty,p}\le(\sin(\pi/p))^{-n}$.  Consider the quantity
\[ p_n := \sup\left\{p\ge2:\,\|P_n^+\|_{\infty,p}\leq1\right\},\] which we following \cite{FIP}
call the critical exponent of $P_n^+$. 
 The critical exponent is well-defined since clearly 
$\|P_n^+\|_{\infty,2}=1$. We  also set 
\[p_\infty := \sup\left\{p\ge2:\,\|P_\infty^+\|_{\infty,p}\leq1\right\}. \] Defining $A_m f(z_1,z_2,\ldots):=f(z_1,\ldots, z_m,0,0,\ldots)$ and using that $\|A_m f\|_p\to \| f\|_p$ as $m\to \infty$ for every $f$ in $L^p(\T^\infty)$ and $1\leq p\leq \infty$, we see that in fact
\[
p_\infty= \lim_{n\to\infty} p_n.
\]
This also follows from the proof of Theorem \ref{n_to_infty} below.

Marzo and Seip \cite{MS} proved that the critical exponent of $P_1^+$ is $4$ and moreover that
\[ 2+2/(2^n-1)\le p_n< 3.67632\] for $n>1$. Recently, Brevig \cite{B} showed that $\lim_{n\to \infty} p_n \le 3.31138$. The following theorem 
settles the asymptotic behavior of the critical exponent of $P_n^+$ when $n\to\infty$.
\begin{theorem}\label{n_to_infty} We have $p_\infty=\lim_{n\to\infty} p_n=2$. 
\end{theorem}
By considering a product of functions in disjoint variables, we obtain the following immediate consequence concerning
the Riesz projection $P_{\infty}^{+}$ on the infinite-dimensional torus, formally defined as 
\[ P_{\infty}^{+}\Big(\sum_{k\in \Z^{(\infty)}}c_\alpha z^\alpha\Big):=\sum_{\alpha\in \N^{(\infty)}}c_\alpha z^\alpha. \]
\begin{corollary}\label{inf_dim}  The Riesz projection $P_{\infty}^{+}$  is  not bounded from $L^q$ to $L^p$ when 
$2<p<q\le \infty$.
\end{corollary}
In turn, since the ``analytic'' dual of $H^1$ obviously equals $P^+_\infty(L^\infty(\T^\infty))$, we obtain a further interesting consequence.
\begin{corollary}\label{cor:dual_non_embedding}
The dual space $H^1(\T^\infty)^*$ is not  contained in $H^p(\T^\infty)$ for any $p>2.$
\end{corollary}
The latter result has an immediate translation in terms of Hardy spaces of Dirichlet series, as will be recorded in
Corollary \ref{cor:dual_non_embedding2} below.

The proof of  Theorem~\ref{n_to_infty} deals with the (pre)dual operator $P^+_\infty:L^q(\T^n)\to L^1(\T^n)$, where $q<2.$ The idea is to prove first that
for the characteristic function  $\chi_A$ of a bounded convex domain $A$ in  $\RR^n$, the norm of the Fourier multiplier $M_{\chi_{A}}$ on $\T^n$ is actually bounded by that of $P^+_{n+m}$ for large enough $m$, depending on $A$. This key observation will be applied when $A$ is a large ball $B(0,R)$ in $\RR^n$, and the desired result is deduced by invoking the following  result of  Ilyin \cite{I1}. 
\begin{theorem}\label{thm:babenko}
The circular Dirichlet kernel
\[
D_{R,n}(\zeta)  := \sum_{\alpha\in \Z^n:\,\|\alpha\|\le R}   \zeta^{\alpha}
\]
on $\T^n$ satisfies  $\| D_{R,n}\|_{L^1(\T^n)}\geq cR^{(n-1)/2},$ where $c=c(n)>0$ and $\|\cdot\|$ stands for the standard Euclidean norm.
\end{theorem}
Babenko's  famous 1971 preprint (see \cite{BA,Li2}) gives another proof. Moreover, it establishes a comparable upper bound, which can also be found in Ilyin and Alimov's paper \cite{I2}. We refer to  Liflyand's review \cite{Li1}  for further information on the related literature and for a  simple proof of Theorem \ref{thm:babenko}.
\begin{proof}[Proof of Theorem \ref{n_to_infty}.]
Fix $n\geq 2$ and  $\alpha =(\alpha_1,\dots,\alpha_n)\in\Z^n$  together with $\beta^j\in\Z^n$ and $ b_j\in\Z$ for $j=1,\dots,m$, where $m\in\N$ is also fixed. We consider
$n+m$ linear functions $\phi_j:\,\Z^n\to\Z$, with  $j=1,\dots,n+m$, where
\begin{align*}
\phi_j(\alpha) & := \alpha_j, \quad j=1,\dots,n, \\
\phi_{n+j}(\alpha) &:= (\alpha,\beta^j) + b_j,\quad j=1,\dots,m. \end{align*}

We associate with any trigonometric polynomial $f$ as in \eqref{def_expan} (that is, any $f$ of the form \eqref{def_expan} with finitely many non-zero terms) 
the function
\[ g(\eta) := \sum_{\alpha\in\Z^n} \hat f(\alpha) \prod_{j=1}^{n+m}
\eta_j^{\phi_j(\alpha)},\]
where $\eta = (\eta_1,\dots,\eta_{n+m})\in\T^{n+m}$.

\begin{lemma}\label{equal_norms} We have $\|g\|_p = \|f\|_p$ for $0<p\le \infty$.
\end{lemma}

\begin{proof} Set
\[ \eta' :=(\eta_1,\dots,\eta_n),\quad \eta'' :=(\eta_{n+1},\dots,\eta_{n+m}).\]
We have
\[ g(\eta) = \psi_0(\eta'') \sum_{\alpha\in\Z^n} \hat f(\alpha)\prod_{j=1}^n (\psi_j(\eta'')\eta_j)^{\alpha_j},\]
where  
\[
\psi_0(\eta''):=\prod_{k=1}^m\eta_{n+k}^{b_k}\; ,\quad\textrm{and}\quad  \psi_j(\eta''):=\prod_{k=1}^m\eta_{n+k}^{\beta^{k}_j}\;\;\textrm{for}\;\; j=1,\ldots, n.
\]
We clearly have $\psi_j(\eta'')\in\T$ 
for $j=0,\dots,n$. For a fixed $\eta''$ in $\T^m$ 
 consider $g$ as a function of $\eta'$:
\[ g(\eta) =g_{\eta''}(\eta').\]
Set $\tilde\eta'=(\tilde\eta_1,\dots,\tilde\eta_n)$, where
$\tilde\eta_j = \psi_j(\eta'')\eta_j$ for $j=1,\dots,n$. We see that
\[ g_{\eta''}(\eta') = \psi_0(\eta'') f(\tilde\eta').\]
We therefore obtain the asserted isometry:
\[ \|g_{\eta''}\|_p = \|f\|_p \]
\end{proof}

By duality, for any positive integer $N$ and $p>2$, we have 
$\|P_N^+\|_{\infty,p} = \|P_N^+\|_{p',1}$ where 
$p'=p/(p-1)$. Hence, to prove Theorem \ref{n_to_infty},
we have to show that for any $q$ in $(1,2)$ there exist a positive integer $N$
and $g$ in $L^q(\T^N)$ such that 
\begin{equation}
\label{refor_Theorem}
\|g\|_q=1,\quad \|P_N^+ g\|_1 >1.
\end{equation}
Indeed, by duality, this will imply the existence of a function $h$ in $L^\infty(\T^n)$
such that 
\[ \|h\|_\infty = 1,\quad \|P_N^+(h)\|_{q'} >1,\]
where $q' = q/(q-1)$. Since $q<2$ is arbitrary, Theorem~\ref{n_to_infty} then follows.

For a bounded set $E$ in $\mathbb{R}^n$ and a function $f$ in $L^{1}(\T^n)$, we consider 
a partial sum of the Fourier series of $f$:
\[ \big(S_Ef\big)(\zeta) := \sum_{\alpha\in E\cap\Z^n} \hat f(\alpha) \zeta^{\alpha}.\]
Note that  as an operator, $S_E$ coincides with  the Fourier multiplier $M_{\chi_E}$.
We say that a polytope $E$ in $\mathbb{R}^n$ is non-degenerate if it is not contained in
a hyperplane.

\begin{lemma}\label{reduc_polyt} Let $1<q<2$. Assume that there is a 
non-degenerate convex polytope $E$ in $\mathbb{R}^n$ with integral vertices such 
that, for some $f$ in $L^q(\T^n)$ with a finite set of non-zero Fourier 
coefficients $\hat f(\alpha)$, we have
\[ \|f\|_q=1,\quad \|S_E(f)\|_1 >1.\]
Then there are a positive integer $N\in\N$ and a function 
$g$ in $L^q(\T^N)$ satisfying \eqref{refor_Theorem}.
\end{lemma}

\begin{proof} Let $e:=(1,1,\ldots, 1)\in\Z^+_n$. By considering instead  $E+Ne$ and $(\eta_1\ldots \eta_n)^Nf(\eta)$ with large enough $N\in \N$, if necessary, we may assume that $E$ and the Fourier coefficients of $f$ satisfy
\begin{equation}
\label{assum_posit}
E\subset\Z_n^+\quad\textrm{and}\quad \hat f(\alpha)\not=0 \;\;\Rightarrow\;\; \alpha_j\geq 0 \;\,\textrm{for all}\;\;j=1,\ldots , n.
\end{equation}
It is known that $E$ is the intersection of closed semispaces,
bounded by the hyperplanes containing the faces of $E$ of dimension $n-1$
(see \cite[Ch. 1, Thm. 5.6]{L}). All hyperplanes are determined by
their intersections with the set of the vertices of $E$. Since the 
vertices are integral,  the semispaces can be defined by inequalities
\[ (\alpha,\beta^j) + b_j\ge0,\quad j=1,\dots,m,\]
where $\beta^j\in\Z^n, b_j\in\Z$ for $j=1,\dots,m$.
Thus
\[ E=\bigcap_{j=n+1}^{n+m} \{\alpha\in\mathbb{R}^m:\,\phi_{j}(\alpha)\ge0\},\]
where $\phi_{j}(\alpha) = (\alpha,\beta^{j-n}) + b_{j-n},\quad j=n+1,\dots,n+m$.

We set $N:=n+m$ and construct the function $g$ from $f$ as in Lemma~\ref{equal_norms}. 
Using that lemma, we get
\[ \|g\|_q = \|f\|_q =1,\quad \|P_N^+ g\|_1 = \|S_E(f)\|_1 >1, \]
and Lemma \ref{reduc_polyt} follows.  \end{proof}

To construct an integer $n$, a polytope $E$, and a function $f$
satisfying Lemma \ref{reduc_polyt}, we take first $n$ satisfying
the inequality
\begin{equation}
\label{choice_n}
n>q/(2-q).
\end{equation}
For sufficiently large $R$, let $E$ be the convex hull of the integral 
points contained in the euclidean ball $\{\alpha\in\mathbb{R}^n:\,\|\alpha\|\le R\}$. Hence for any function $f$ in $L^{1}(\T^n)$,
we have
\[ (S_Ef)(\zeta) = \sum_{\alpha\in \Z^n:\,\|\alpha\|\le R} \hat f(\alpha) \zeta^{\alpha}.\]

Recall the circular Dirichlet kernel from Theorem \ref{thm:babenko}:
\[ D_{R,n}(\zeta) = \sum_{\alpha\in \Z^n:\,\|\alpha\|\le R}  \zeta^{\alpha}.\]
Define the function
$\displaystyle \widetilde f (\zeta):=\sum_{|\alpha_1|\le R}\dots \sum_{|\alpha_n|\le R}\zeta^\alpha$
so that
$S_R \widetilde f = D_{R,n}.$
It is easy to see that
\[ \|\widetilde f\|_q = \left\|\sum_{|\alpha_1|\le R} \zeta_1^{\alpha_1}\right\|_q^n
\le C R^{n(1-1/q)},\]
where $C= C(q,n) >0$. In view of \eqref{choice_n}, which amounts to $\frac{n-1}{2}>n(1-\frac{1}{q}),$ and by recalling Theorem \ref {thm:babenko}, we obtain
\[ \|S_E(\widetilde f)\|_1 > \|\tilde f\|_q. \]
for sufficiently large $R$. Taking
$f := \widetilde f\big/\|\widetilde f\|_q,$
we get a function $f$ satisfying the conditions of Lemma \ref{reduc_polyt}, and
this completes the proof of Theorem \ref{n_to_infty}. 
\end{proof}

\section{The space of Dirichlet series in $\BMOA$}\label{sec:BMOA}
The result of the preceding section is purely multiplicative in the sense that it only involves analysis on the product space $\T^n$. Function spaces
on $\T^{n}$ or on $\T^{\infty}$ may however, by a device known as the Bohr lift (see below for details), also be viewed as spaces of Dirichlet series. From an abstract point of view (see for example \cite[Ch. 8]{R}), this means that we equip our function spaces with an additive structure that reflects the additive order of the multiplicative group of positive rational numbers $\mathbb{Q}_+$. This results in interesting interaction between function theory in polydiscs and half-planes that sometimes involves nontrivial number theory.

As we will see in the next subsection, this point of view leads us naturally from $H^1(\T^{\infty})^*$ to the space of ordinary Dirichlet series $\sum_{n=1}^{\infty} a_n n^{-s}$ that belong to $\BMOA$, i.e., the space of analytic functions  
$f(s)$ in the right half-plane $\Real s> 0$ satisfying
\begin{equation} \label{eq:int1} \sup_{\sigma>0} \int_{-\infty}^{\infty} \frac{|f(\sigma+it)|^2}{1+\sigma^2+t^2} dt < \infty \end{equation} and
\[\|f\|_{\BMO} := \sup_{I \subset \mathbb{R}} \frac{1}{|I|}\int_{I}\left|f(it)-\frac{1}{|I|}\int_I f(i\tau)\,d\tau\right|\,dt < \infty.\]
Here the supremum is taken over all finite intervals $I$; \eqref{eq:int1} means that $g(s):=f(s)/(s+1)$ belongs to the Hardy space $H^{2}(\C_0)$ of the right half-plane $\mathbb{C}_0$, and then
$f(it):=\lim_{\sigma\to 0^+} f(\sigma+it)$ exists for almost all real $t$ by Fatou's theorem applied to $g$.  We will use the notation $\BMOA \cap \mathcal{D}$ for this $\BMOA$ space, where $\mathcal{D}$ is the class of functions expressible as a convergent Dirichlet series in some half-plane $\Real s > \sigma_0$. 


The space $\BMOA \cap \mathcal{D}$ arose naturally in a recent study of multiplicative Volterra operators \cite{BPS}. We refer  to that paper for a complementary discussion of bounded mean oscillation in the context of Dirichlet series. By combining \cite[Cor. 6.4]{BPS} and \cite[Thm. 5.3]{BPS}, we may conclude that $\BMOA \cap \mathcal{D}$ can be viewed, via the Bohr lift, as a subspace of $H^{1}(\T^{\infty})^\ast$. This inclusion may however be proved in a direct way by an argument that we will present in the next subsection.
\subsection{The Bohr lift and the inclusion $\BMOA \cap \mathcal{D}\subset (\mathcal{H}^1)^*$}
We begin by considering an ordinary Dirichlet series of the form
\begin{equation} \label{eq:f} f(s)=\sum_{n=1}^\infty a_n n^{-s}. \end{equation}
By the transformation $z_j=p_j^{-s}$ (here $p_j$ is the $j$th prime number) and the fundamental theorem of arithmetic, we have the Bohr correspondence, 
\begin{equation}\label{eq:bohr} 
	f(s):= \sum_{n=1}^\infty a_{n} n^{-s}\quad\longleftrightarrow\quad \mathcal{B}f(z):=\sum_{n=1}^{\infty} a_n z^{\kappa(n)}, 
\end{equation}
where $\kappa(n)=(\kappa_1,\ldots,\kappa_j,0,0,\ldots)$ is the multi-index such that $n = p_1^{\kappa_1} \cdots p_j^{\kappa_j}$. The transformation $\mathcal{B}$ is known as the Bohr lift. For $0 < p < \infty$, we define $\mathcal{H}^p$ as the space of Dirichlet series $f$ such that $\mathcal{B}f$ is in $H^p(\mathbb{T}^\infty)$, and we set 
\[\|f\|_{\mathcal{H}^p} := \|\mathcal{B}f\|_{H^p(\mathbb{T}^\infty)} = \left(\int_{\mathbb{T}^\infty} |\mathcal{B}f(z)|^p\,dm_\infty(z)\right)^\frac{1}{p}.\]
Note that for $p=2$, we have
\[\|f\|_{\mathcal{H}^2} = \left(\sum_{n=1}^\infty |a_n|^2\right)^\frac{1}{2}.\]
In terms of the spaces ${\mathcal{H}^p}$, Corollary \ref{cor:dual_non_embedding} takes the form
\begin{corollary}\label{cor:dual_non_embedding2}
The dual space $(\mathscr{H}^1)^*$ is not  contained
in $\mathscr{H}^p$ for any $p>2.$
\end{corollary}
We will now use the notation $\mathbb{C}_{\theta}:=\{s=\sigma+it: \sigma>\theta\}$. The conformally invariant Hardy space $H_{\operatorname{i}}^p(\mathbb{C}_\theta)$ consists of functions $f$ that are analytic on $\mathbb{C}_\theta$ and satisfy 
\[\|f\|_{H_{\operatorname{i}}^p(\mathbb{C}_\theta)} := \sup_{\sigma>\theta} \left(\frac{1}{\pi}\int_{\mathbb{R}} |f(\sigma+it)|^p\,\frac{dt}{1+t^2}\right)^\frac{1}{p} <\infty.\]
These spaces show up naturally in our discussion in the following two ways. First, we will repeatedly use that a function $g$ analytic on $\mathbb{C}_0$ is in $\BMOA$ if and only if the measure
\[ d\mu(s):=|g'(\sigma+it)|^2 \sigma d\sigma  \frac{ dt}{1+t^2} \] 
is a Carleson measure for $H_{\operatorname{i}}^1(\mathbb{C}_0)$, which means that there is a constant $C$ such 
\[ \int_{\mathbb{C}_0} |f(s)| d\mu(s) \le C \| f \|_{H_{\operatorname{i}}^1(\mathbb{C}_0)} \]
for all $f$ in $H_{\operatorname{i}}^1(\mathbb{C}_0)$. The smallest such constant $C$ is called the Carleson norm of the measure.
Second, by Fubini's theorem, we have the following connection between $\mathcal{H}^p$ and $H_{{\operatorname{i}}}^p(\mathbb{C}_0)$:
\begin{equation}
	\label{eq:avgrotemb} \big\|f\big\|_{\mathcal{H}^p}^p = \int_{\mathbb{T}^\infty} \|f_\chi\|^p_{H^p_{\operatorname{i}}(\mathbb{C}_0)} \, dm_\infty(\chi),
\end{equation}
where $\chi$ is a character on $\mathbb{Q}^+$, i.e., a completely multiplicative function taking only unimodular values, and
\[ f_{\chi}(s):=\sum_{n=1}^{\infty} \chi(n) a_n n^{-s}. \]  Here we recall that an arithmetic function $g:\N\to\C$ is completely multiplicative if it satisfies $g(nm)$ $=g(n)g(m)$ for all integers $m,n\geq 1$. A completely multiplicative function $g$ satisfies $g(1)=1$ unless $g$ vanishes identically, and it is completely determined by its values at the primes. 

Note that we identify via the Bohr lift $\alpha\mapsto p^{\alpha}$ the  group $\Z^{(\infty)}$ with the group $Q^+$, and by duality  the group $\T^\infty$ with the group of completely multiplicative functions $\chi:\N\to \T$. Accordingly, we identify the Haar measures $dm_{\infty}(z)$ and $dm_{\infty}(\chi)$ of both groups. We also used in \eqref{eq:avgrotemb}  the fact that, for almost every character $\chi$  and $\sum_{n=1}^\infty a_n n^{-s}$ in $\mathcal{H}^p$, the series $\sum_{n=1}^\infty a_n \chi(n) n^{-s}$ converges $m_{\infty}$-almost everywhere in $\C_0$, and defines an element in $H_{{\operatorname{i}}}^p(\mathbb{C}_0)$. For these facts, we refer e.g. to \cite[Section 4.2]{HLS} and \cite[Thm 5]{Ba}.

From \eqref{eq:avgrotemb} we may deduce Littlewood--Paley type expressions for the norms of $\mathcal{H}^p$. This was first done for $p=2$ in \cite[Prop.~4]{Ba}, and later for $0<p < \infty$ in \cite[Thm.~5.1]{BQS}, where the formula
\begin{equation}
	\label{eq:LPp} \|f\|_{\mathcal{H}^p}^p \asymp |f(+\infty)|^p + \frac{4}{\pi}\int_{\mathbb{T}^\infty} \int_{\mathbb{R}}\int_0^\infty |f_\chi(\sigma+it)|^{p-2}|f_\chi'(\sigma+it)|^2 \sigma d\sigma\frac{dt}{1+t^2}dm_\infty(\chi)
\end{equation}
was obtained. When $p=2$, we have equality between the two sides of \eqref{eq:LPp}.

The Littlewood--Paley formula \eqref{eq:LPp}  for $p=2$ may be polarized, so that we have
	\[
		\langle f, g \rangle_{\mathcal{H}^2} = f(+\infty)\overline{g(+\infty)} + \frac{4}{\pi} \int_{\mathbb{T}^\infty} \int_\mathbb{R} \int_0^\infty f_\chi'(\sigma+it) \overline{g_\chi'(\sigma+it)} \sigma\,d\sigma\,\frac{dt}{1+t^2}\,dm_\infty(\chi). \label{eq:LPpolar} 
	\]
Hence, by the Cauchy--Schwarz inequality and \eqref{eq:LPp}, we have for $f$ in $\mathcal{H}^1$ and $g$ in $\BMOA\cap \mathcal{D}$,
\begin{align*}
\big|\langle f, g \rangle_{\mathcal{H}^2} - f(+\infty)\overline{g(+\infty)}\big|^2 & \le \frac{4}{\pi}\int_{\mathbb{T}^\infty} \int_\mathbb{R} \int_0^\infty |f_\chi(\sigma+it)|^{-1} |f_\chi'(\sigma+it)|^2 \sigma\,d\sigma\,\frac{dt}{1+t^2}\,dm_\infty(\chi) \\
& \times 
\int_{\mathbb{T}^\infty} \int_\mathbb{R} \int_0^\infty |f_\chi(\sigma+it)| |g_\chi'(\sigma+it)|^2 \sigma\,d\sigma\,\frac{dt}{1+t^2}\,dm_\infty(\chi) \\
& \ll \| f \|_{\mathcal{H}^1} \int_{\T^{\infty}}\|f_\chi\|_{H^1_{\operatorname{i}}(\mathbb{C}_0)} \, dm_\infty(\chi) = \| f \|_{\mathcal{H}^1}^2,
\end{align*}
where we in the second step used the Littlewood--Paley formula for $p=1$ and that 
\[ |g'_{\chi}(\sigma+it)|^2 \sigma d\sigma  \frac{ dt}{1+t^2} \] 
is a Carleson measure for $H_{\operatorname{i}}^1(\mathbb{C}_0)$, with Carleson constant uniformly bounded in $\chi$, as follows from \cite[Lem. 2.1 (ii) and Lem. 2.2]{BPS}.
Hence we conclude that a Dirichlet series $g$ in $\BMOA\cap \mathcal{D}$ belongs to
$(\mathcal{H}^1)^*$.

The ``reverse'' problem\label{brevig} of finding an embedding of $(\mathcal{H}^1)^*$ into a ``natural'' space of functions analytic in $\mathbb{C}_{1/2}$ appears challenging. (This is a  reverse question only in a rather loose sense as we are now considering functions defined in $\mathbb{C}_{1/2}$.) It was mentioned in \cite[Quest. 4]{SaSe} that $(\mathcal{H}^1)^*$ is not contained in $H_{\operatorname{i}}^q(\mathbb{C}_{1/2})$
for any $q>4$. Since no argument for this assertion was given in \cite{SaSe}, we take this opportunity to offer a proof\footnote{We thank Ole Fredrik Brevig for showing us this argument and allowing us to include it in this paper.}. To begin with, let us consider the interval from $1/2-i$ to $1/2+i$ and let $E$ denote the corresponding local embedding of $\mathscr{H}^2$ into $L^2(-1,1)$, given by $Ef(t) := f(1/2+it)$, so that
\[\|E f\|_{L^2(-1,1)}^2 = \int_{-1}^1 |f(1/2+it)|^2 \,dt.\]
Then the adjoint $E^\ast \colon L^2(-1,1) \to \mathscr{H}^2$ is 
\[E^{\ast}g(s) := \sum_{n=1}^\infty \frac{\widehat{g}(\log{n})}{\sqrt{n}} n^{-s},\]
where $\widehat{g}(\xi) = \int_{-1}^1 e^{-i \xi t} g(t)\,dt$.
Fix $0<\beta<1$ and set $g_\beta(t): = |t|^{\beta-1}$. Plainly, $g_\beta$ is in $L^q(-1,1)$ if and only if $\beta>1-1/q$. Moreover, if $\xi\geq \delta>0$, then  
$\widehat g_\beta(\xi) \asymp \xi^{-\beta}$, 
where the implied constants depend only on $\delta$ and $\beta$. We now invoke Helson's inequality \cite[p. 89]{He2}
\[ \Big\| \sum_{n=1}^{\infty} a_n n^{-s} \Big\|_1 \ge \left(\sum_{n=1}^{\infty} \frac{|a_n|^2}{d(n)} \right)^{1/2}, \]
where $d(n)$ is the divisor function. We then use the classical fact that $\sum_{n\le x} 1/d(n)$ is of size $x (\log x)^{-1/2}$; the precise asymptotics of this summatory function was first computed by Wilson \cite[Formula (3.10)]{W} and may now be obtained as a simple consequence of a general formula of Selberg \cite{Se}. Taking $\beta=1/4$, we may therefore infer by partial summation that $E^\ast$ is unbounded from $L^q(-1,1)$ to $\mathscr{H}^1$ whenever $q < 4/3$.
By duality we conclude that for any $q>4$, there are $\varphi$ in $(\mathscr{H}^1)^\ast$ that are not locally embedded in $L^q(-1,1)$  and hence do not belong to $H_{\operatorname{i}}^q(\mathbb{C}_{1/2})$. Note that here  $(\mathscr{H}^1)^\ast$ is identified as a subspace of $H^2$ (with respect to the natural pairings of $L^2(-1,1)$ and $L^2(\mathbb{T}^\infty)$) ,whence $E^{**}g=Eg$ for $(\mathscr{H}^1)^\ast$.
In view of  Corollary~\ref{cor:dual_non_embedding}, it is natural to ask if the situation is even worse, namely that $(\mathcal{H}^1)^*$ fails to be contained in $H_{\operatorname{i}}^q(\mathbb{C}_{1/2})$ for any $q>2$. 

We conclude from the preceding argument that there is no simple relation between $(\mathcal{H}^1)^*$ and $\BMOA(\mathbb{C}_{1/2})$. We may further illustrate this point by the following example. The Dirichlet series
\[ h(s):=\sum_{n=2}^{\infty} \frac{1}{\log n} n^{-s-1/2} \]
belongs to $\BMOA(\mathbb{C}_{1/2})$ (see \eqref{eq:hilbert}) below), but it is unknown whether it is in $(\mathcal{H}^1)^*$. It would be interesting to settle this question about membership in $(\mathcal{H}^1)^*$, as $h$ is both a primitive of $\zeta(s+1/2)-1$ and the analytic symbol of the multiplicative Hilbert matrix \cite{BPSSV}. 
\subsection{Fefferman's condition for membership in $\BMOA\cap \mathcal{D}$}
The following theorem gives interesting information about Dirichlet series in $\BMOA$. It is an immediate consequence of existing results, as will be explained in the subsequent discussion.
\begin{theorem}  \label{thm:fefferman}
\begin{itemize}
\item[(i)]
Suppose that $a_n\ge 0$ for every $n\ge 1$. Then $f(s):=\sum_{n=1}^{\infty} a_n n^{-s}$ is in $\operatorname{BMOA}$  if and only if
\begin{equation} \label{eq:feff} S^2:=\sup_{x\ge e } \sum_{k=1}^{\infty} \Big(\sum_{x^k\le n <x^{k+1}} a_{n} \Big)^2 < \infty,  \end{equation}
and we have $S\asymp \Vert f\Vert_{\BMOA}$.
\item[(ii)]
If \ $\sum_{n=1}^{\infty} |a_n| n^{-s}$ is in $\operatorname{BMOA}$, then $\sum_{n=1}^{\infty} a_n n^{-s}$ is in $\operatorname{BMOA}$.
\end{itemize}
\end{theorem}
It is immediate from (i) that 
\begin{equation} \label{eq:hilbert} \sum_{n=2}^{\infty} \frac{1}{\log n} n^{-s-1} \end{equation}
is in $\BMOA$ (see \cite[Thm. 2.5]{BPS}). By Mertens's formula 
\begin{equation} \label{eq:mertens} \sum_{p\le x} \frac{1}{p} = \log\log x + M +O\left((\log x)^{-1}\right), \end{equation}
where the sum is over the primes $p$, part (i) also implies that 
$ \sum_p p^{-1-s} $
is in $\BMOA$, and consequently $\log \zeta(s+1)$ is a function in $\BMOA$, where $\zeta(s)$ is now the Riemann zeta function. Then  part (ii) of Theorem \ref{thm:fefferman}  implies also that
$\sum_p \chi(p) p^{-1-s}$ is in $\BMOA$ for any sequence of unimodular numbers  $\chi(p)$. In fact, we have more generally:

\begin{corollary}\label{cor:prime}
A Dirichlet series $\sum_{p} a_p p^{-s}$ over the primes $p$ is in $\operatorname{BMOA}$ if and only if
\begin{equation} \label{eq:bmo}  \sup_{x\ge e} \sum_{k=1}^{\infty} \Big(\sum_{x^k\le p < x^{k+1}} |a_{p}| \Big)^2 < \infty.\end{equation}
\end{corollary}

\noindent Corollary~\ref{cor:prime} is a consequence of part (i) of Theorem~\ref{thm:fefferman} and the fact (see \cite[Lem. 2.1]{BPS}) that $\sum_p a_p p^{-s}$ is in $\BMOA$ if and only if $\sum_p a_p \chi(p) p^{-s}$ is in $\BMOA$ for every sequence of unimodular numbers  $\chi(p)$.  

The sufficiency of condition \eqref{eq:feff} in Theorem~\ref{thm:fefferman})(i) follows as a corollary to an $H^1$ multiplier theorem of Sledd and Stegenga \cite[Thm. 1]{SS} via Fefferman's duality theorem \cite{F, FS} and Parseval's theorem. The necessity also follows from \cite[Thm. 1]{SS} if we first note that  for any $f$ in $H^1(\C_0)$, using the standard $H^2$ factorization of $H^1$, we may construct $g$ in $H^1(\C_0)$ with $\| g\|_{H^1(\C_0)}=\| f\|_{H^1(\C_0)}$  and   $\widehat g(\xi)\geq  |\widehat f(\xi)|\geq 0$ for all $\xi\in\RR$. Here $\widehat f,\widehat g$ refer to the Fourier transforms of the boundary values on the imaginary axis. A corresponding result for $\BMO$ in the unit disc is stated  in \cite[Cor. 2]{SS}: The Taylor series $\sum_{m=0}^{\infty} c_m z^m$ with $c_m\ge 0$ belongs to $\BMO$ of the unit circle $\mathbb{T}$ if and only if  
\[ \sup_{m\ge 1} \sum_{j=0}^{\infty} \left(\sum_{r=0}^{m-1} c_{mj+r}\right)^2 < \infty. \]
Other proofs of this result, relying more directly on Hankel operators, can be found in \cite{Bon, HW}. This result is commonly known to have appeared in unpublished work of Fefferman.

To establish part (ii) of Theorem~\ref{thm:fefferman}, we use the following Carleson measure characterization of $\BMOA\cap \mathcal{D}$ which could be used to give an alternative proof of part (i) of Theorem~\ref{thm:fefferman}.

\begin{lemma}\label{basaux} Suppose that $f$ is in $H_{\operatorname{i}}^{2}(\C_0)\cap \mathcal{D}$. Then $f$ is in $\BMOA\cap\mathcal{D}$  if and only if there exists a positive constant $C$ such
\begin{equation} \label{eq:carl}  \sup_{t\in \mathbb{R}} \int_{0}^h \int_{t}^{t+h} |f'(\sigma+i \tau)|^2 \sigma d\tau d\sigma \le C h \end{equation}
for $0\le h \le 1$.
Moreover, the best constant $C$ in ( \ref{eq:carl}) and $\Vert f\Vert_{\BMO}^{2}$ are equivalent.
\end{lemma}
\begin{proof} We first observe that \eqref{eq:carl} and the assumption that $f$ is in $H_{i}^{2}(\C_0)$ imply, by the maximum modulus principle, that $f'(\sigma+it)$ is uniformly bounded by $O(\sqrt{C})$ for $\sigma\geq 1$. Then, if $h>1$ and $t\in \RR$ are given and 
\begin{align*} I & :=\int_{0}^h \int_{t}^{t+h} |f'(\sigma+i\tau)|^2 \sigma d\tau d\sigma \\ &\ =\int_{0}^{1}\Big[\int_{t}^{t+h}|f'(\sigma+i\tau)|^2 d\tau\Big]\sigma d\sigma+ \int_{1}^{h}\Big[\int_{t}^{t+h}|f'(\sigma+i\tau)|^2 d\tau\Big]\sigma d\sigma=:I_1+I_2,\end{align*}
we  have $I_1\ll Ch$ by \eqref{eq:carl}, while 
\[ I_2 \ll   \int_{1}^{\infty}\Big[\int_{t}^{t+h}|f'(\sigma+i\tau)|^2 d\tau\Big]\sigma d\sigma\ll \int_{t}^{t+h}\Big[\int_{1}^{\infty} \sigma C4^{-\sigma}d\sigma\Big]d\tau \ll Ch.\]
To obtain the final estimate above, we used that $f'(\sigma+it)=O(\sqrt{C} 2^{-\sigma})$, which holds uniformly in $t$ when $\sigma\geq 1$ because $f$ is a Dirichlet series.
\end{proof}

Part (ii) of Theorem~\ref{thm:fefferman} is immediate from this lemma along with a property of almost periodic functions established by  Montgomery \cite[p. 131]{MO} (see also \cite[p. 4]{M}) which asserts that if 
$|a_n|\leq b_n$, then for sums with a finite number of non-zero terms \[ \int_{T_1-T}^{T_1+T} \big|\sum a_n e^{i\lambda_n t}\big|^2dt\leq 3\int_{-T}^T \big|\sum b_n e^{i\lambda_n t}\big|^2dt.\]
Here $T>0$, $T_1$ is a real number, $a_n, b_n$ respectively complex and nonnegative coefficients, 
and $\lambda_n$ are distinct real frequencies.

We will now apply Theorem~\ref{thm:fefferman} to see how our $\BMOA$ space of Dirichlet series relates to Hardy spaces and the Bloch space. We denote as usual $H^{\infty}(\C_0) \cap \mathcal{D}$ by $\mathcal{H}^\infty$, and we say that a function $f(s)$ analytic in $\Real s >0$ is in the Bloch space $\mathfrak{B}$ if 
\[ \| f \|_{\mathfrak{B}}:=\sup_{\sigma+it: \sigma>0} \sigma |f'(\sigma+i t)| < \infty . \] 
We have 
\[ \mathcal{H}^{\infty} \subset \BMOA \cap \mathcal{D} \subset \bigcap_{0<q<\infty} \mathcal{H}^q, \]
where the inclusion to the left is trivial and that to the right was established in \cite[Lem. 2.1]{BPS}. Hence, in contrast to $(\mathcal{H}^1)^*$ itself, the subspace $\BMOA\cap \mathcal{D}$ is included in $ \bigcap_{0<q<\infty} \mathcal{H}^q$. 
Moreover,  is a classical fact and easy to see that $\BMOA \subset \mathfrak{B}$. 

The following consequence of Corollary~\ref{cor:prime} is a Dirichlet series counterpart to a result of Campbell, Cima, and Stephenson \cite{CCS} that further enunciates the relation between the spaces in question. Our proof is close to that found in \cite{HT}.   

 \begin{corollary}\label{cor:bloch}There exist Dirichlet series that belong to $\mathfrak{B}$ and 
 $\bigcap_{0<q<\infty} \mathcal{H}^q$ but not to $\BMOA$.
 \end{corollary}
 \begin{proof}
 It is an easy consequence of the definition of the Bloch space that $\sum_{n=1}^{\infty} a_n n^{-s}$ with $a_n\ge 0$ is in $\mathfrak{B}$ if and only if
\begin{equation}\label{eq:Bloch} \sup_{x\ge 2} \sum_{x\le n <x^2} a_n < \infty. \end{equation}
Indeed, if \eqref{eq:Bloch} holds, then we use it with $x_j=\exp(2^{j}/\sigma),\  x_{j+1}=x_{j}^{2}$, to show that for $\sigma>0$,
\[ \sum_{n\geq 2} a_n\, \sigma \log n\,e^{-\sigma\log n}\le \sum_{j} 2^{-j} \big(\sum_{x_j\leq n<x_{j+1}}a_n\big)\ll  \sum_{j} 2^{-j}.\]
Conversely, if $\sum_{n\geq 2} a_n\, \sigma \log ne^{-\sigma\log n}\leq C$ for all $\sigma>0$, then choosing $\sigma=1/\log x$, we see that the sum on the left-hand side of \eqref{eq:Bloch} 
is bounded by  $C e^2/2$.
Let $\mathbb{P}_j$ be the primes in the interval $[e^{2^j}, e^{2^j+1}]$. Then 
 $|\mathbb{P}_j|\sim (e-1) e^{2^j} 2^{-j}$ by the prime number theorem. Setting $a_p:=e^{-2^j} 2^j$ if $p$ is in $\mathbb{P}_j$ and $a_p=0$ otherwise, we see from \eqref{eq:Bloch} that $\sum_p a_p p^{-s}$ is in the Bloch space, but from part (i) of Theorem~\ref{thm:fefferman} that it fails to be in $\BMOA$. 
 
 We next recall Khinchin's inequality for the Steinhaus variables $Z_p$  (that are i.i.d. random variables with uniform distribution on $\T$): 
 \[
 \mathbb{E} \big| \sum_pa_p Z_p\big|^q\asymp \big(\sum_p|a_p|^2 \big)^{q/2},
 \]
with the implied constants only depending on $q>0$ (see  \cite[Thm. 1]{K}). Since in the Bohr correspondence  $p_k^{-s}$ corresponds to the independent variable $z_k$, we see that they form a sequence of Steinhaus variables with respect to the Haar measure on $\T^\infty$. 
Thus, in view of the bound
 \[ \sum_p a_p^2 \ll \sum_{j=0}^{\infty} e^{-2^j} 2^j < \infty, \]
Khinchin's inequality implies that 
 $\sum_p a_p p^{-s}$  belongs to $\mathcal{H}^q$.
 \end{proof}


\subsection{The relation between Dirichlet series in $H^{\infty}$, $\BMOA$, and $\mathfrak{B}$}
\label{sec:relation}

We turn to some further comparisons between the three spaces $\mathcal{H}^{\infty}$,  $\BMOA\cap \mathcal{D}$, and $\mathfrak{B}\cap \mathcal{D}$. We begin with a discussion of uniform and absolute convergence of Dirichlet series in $\mathfrak{B}\cap \mathcal{D}$.  The following lemma will be useful in this discussion. Here we use the notation $\log_+ x:=\max(0,\log x)$ for $x>0$, and we will also write $(T_c f)(s):=f(s+c)$ in what follows.
\begin{lemma}\label{vari}Suppose that $f(s)=\sum_{n=1}^\infty a_n n^{-s}$  is in  $\mathfrak{B}\cap \mathcal{D}$. Then 
\begin{align}
\label{eq:coeff} |a_n| & \le e \| f \|_{\mathfrak{B}}, \quad n\ge 2, \\ \label{eq:point}
|f(\sigma+it)-a_1| & \le \left(\log_+\frac{1}{\sigma} + C 2^{-\sigma} \right) \| f \|_{\mathfrak{B}}, \quad \sigma>0, \end{align}
for some absolute constant $C$. Up to the precise value of $C$, these bounds are both optimal.
\end{lemma}
\begin{proof} To prove \eqref{eq:coeff}, we use that $T_{\varepsilon}f'$ is in $\Hi$ for every  $\varepsilon>0$. By either viewing the coefficients of a Dirichlet series as Fourier coefficients or using that $\| f \|_{\mathcal{H}^2}\le \| f \|_{\mathcal{H}^{\infty}}$, we see that they are dominated by its $\mathcal{H}^\infty$ norm. We therefore have \[ |a_n|(\log n) n^{-\varepsilon}\leq \Vert T_{\varepsilon}f' \Vert_\infty \leq \frac{\Vert f\Vert_{\mathfrak{B}}}{\varepsilon} \] 
and hence 
\[ |a_n|\leq \frac{n^{\varepsilon}\Vert f\Vert_{\mathfrak{B}}}{\varepsilon \log n}.\]
We conclude by taking $\varepsilon=1/\log n$. In addition, we notice that the bound is optimal because $\| n^{-s} \|_{\mathfrak{B}}=1/e$. 

To prove \eqref{eq:point}, we begin by noticing that \eqref{eq:coeff} implies that 
\begin{equation} \label{eq:large} |f(\sigma+it)-a_1| \le \sum_{n=2}^{\infty} |a_n| n^{-\sigma} \le e (\zeta(\sigma)-1) \| f \|_{\mathfrak{B}} \end{equation}
holds for $\sigma\ge 2$.
For $\sigma\le 2$, we use that 
\[ |f(\sigma+it)-a_1| \le |f(2+it)-a_1|+ \int_{\sigma}^2 \| f \|_{\mathfrak{B}} \frac{d\alpha}{\alpha} \le \left(\log \frac{1}{\sigma} + C\right)\| f \|_{\mathfrak{B}},  \]
where we in the final step used \eqref{eq:large} with $\sigma=2$. The example $\sum_{n=2}^{\infty} n^{-1-s}/\log n$ shows that the inequality is optimal, up to the precise value of $C$.
\end{proof}
The pointwise bound \eqref{eq:point} implies that what is known about uniform and absolute convergence of Dirichlet series in 
$\mathcal{H}^{\infty}$ carries over in a painless way to $\mathfrak{B}\cap \mathcal{D}$. In fact, a rather weak bound of the form
\begin{equation} \label{eq:gen}  |f(\sigma+it)|\le C(\sigma), \quad \sigma>0, \end{equation}
suffices to draw such a conclusion, as will now be explained. To begin with we will assume that $C(\sigma)$ is an arbitrary positive function and later specify its required behavior as $\sigma\to 0^+$.

First, by a classical theorem of Bohr \cite[p. 145]{MAHE}, a bound like \eqref{eq:gen} implies that
the Dirichlet series of $f(s)$ converges uniformly in every half-plane $\Real s \ge \sigma_0>0.$ Following Bohr, we then see that 
 $\sigma_{u}(f)\leq 0$, where $\sigma_u(f)$ is the abscissa of uniform convergence, defined as the infimum over those $\sigma_0$ such that the Dirichlet series of $f(s)$ converges uniformly in $\Real s \ge \sigma_0$.

Second, as observed by Bohr, it is immediate that $\sigma_{u}(f)\leq 0$ implies $\sigma_{a}(f)\leq 1/2$, where $\sigma_{a}(f)$ is the abscissa of absolute convergence of $f$, i.e., the infimum over those $\sigma_0$ such that the Dirichlet series of $f(s)$ converges absolutely in $\Real s \ge \sigma_0$. Thanks to more recent work originating in \cite{BCQ}, an interesting refinement of this result holds when $C(\sigma)$ does not grow too fast as $\sigma\searrow 0$. To arrive at that refinement, we set $(S_N f)(s):=\sum_{n=1}^{N} a_n n^{-s} $ and recall that 
\begin{equation}\label{improv2}\sum_{n=1}^N|a_n| \leq \sqrt{N} e^{-c_N \sqrt{\log N \log\log N}}\Vert S_N f\Vert_\infty  \end{equation}
with $c_N\to 1/\sqrt{2}$ when $N\to \infty$. This ``Sidon constant'' estimate was proved in \cite{KQ} with a smaller value of $c_N$. The proof from \cite{KQ}, using at one point the hypercontractive Bohnenblust--Hille inequality from \cite{DFOOS}, yields \eqref{improv2} with
$c_N\to 1/\sqrt{2}$, which is stated as Theorem 3 in \cite{DFOOS}. This is optimal by \cite{dB}.

It was proved in \cite{BCQ} that there exists an absolute constant $C$ such that if $f(s):=\sum_{n=1}^{\infty}{a_nn^{-s}}$ is in $\mathcal{H}^{\infty}$, then 
$ \| S_N f\|_{\infty}\le C \log N \| f\|_{\infty} $. See also Section~\ref{sec:partial}, where an alternate proof of this bound will be given. Using this fact, we obtain from \eqref{improv2} that
\begin{equation}\label{improv3}\sum_{n=1}^N|a_n| \leq \sqrt{N} e^{-c_N \sqrt{\log N \log\log N}}\Vert f \Vert_\infty , \end{equation}
still with $c_N\to 1/\sqrt{2}$ when $N\to \infty$. Now applying \eqref{improv3} to $T_{\varepsilon} f $ with $\varepsilon=1/\log N$ and taking into account 
\eqref{eq:gen}, we get
\[ \sum_{n=1}^N |a_n| \le e \sum_{n=1}^N |a_n| n^{-\varepsilon} \le \sqrt{N} e^{-c_N \sqrt{\log N \log\log N}} C(1/\log N). \]
We now see that if $\log C(\sigma)=o(\sqrt{|\log \sigma |/\sigma })  $ when $\sigma\searrow 0$, then 
\begin{equation} \label{eq:imp}\sum_{n=1}^N |a_n|  \le \sqrt{N} e^{-c_N \sqrt{\log N \log\log N}} \end{equation}
with $c_N\to 1/\sqrt{2}$. When $f$ is in $\mathfrak{B}$, we have $C(\sigma)=O(|\log \sigma|)$ and hence \eqref{eq:imp} clearly holds. Summing by parts and using \eqref{eq:imp}, we get
\begin{equation} \label{eq:sum} \sum_{n=3}^{\infty} \frac{|a_n|}{\sqrt{n}} e^{{c \sqrt{\log n \log\log n}}} < \infty \end{equation}
for every $c<1/\sqrt{2}$. This is a bound previously known to hold for functions $f$ in $\mathcal{H}^{\infty}$ (see  \cite{BCQ, DFOOS}). As shown in \cite{DFOOS}, the result is optimal in the sense that there exist functions $f$ in $\mathcal{H}^{\infty}$ for which the series in \eqref{eq:sum} diverges when $c>1/\sqrt{2}$.
 
In Section~\ref{sec:poly}, we will establish ``reverse''  inequalities to $\| f \|_{\mathfrak{B}} \le \| f \|_{\infty} $ and $\| f \|_{\mathfrak{B}} \ll \| f\|_{\BMOA}$ when $f(s)=\sum_{n=1}^N a_n n^{-s}$ and $N$ is fixed.   
\subsection{A condition for random membership in  $\BMOA \cap \mathcal{D}$}
 In the sequel, if $f(s)=\sum_{n=1}^\infty a_n n^{-s}$ is a Dirichlet series, we denote by $f_\omega$ the corresponding randomized Dirichlet series, namely $f_{\omega}(s):=\sum_{n=1}^\infty \varepsilon_{n}(\omega)a_n n^{-s}$ where $(\varepsilon_n)$ is a standard Rademacher sequence.
 We are interested in extending the following result of Sledd \cite{S} (see also \cite{DUR}) to the setting of ordinary Dirichlet series:
 \begin{theorem}\label{dur}Suppose $\sum_{n=1}^\infty |a_n|^2 \log n<\infty$. Then, the power series $\sum \varepsilon_n a_n z^n$ is almost surely in 
$\BMOA$. \end{theorem}  
This result is optimal in a rather strong sense as shown in \cite{ACP}: If one replaces $\log n$ by any sequence growing at a slower rate, then the condition does not  guarantee membership even in the Bloch space.

We see from Theorem~\ref{dur} that if we require slightly more than $\ell^2$ decay of the coefficients, then we may expect that a ``generic'' analytic function in the unit disc will be in $\BMOA$. The results of the preceding sections show in two respects that a similarly strong result can not hold in the context of Hardy spaces of Dirichlet series. 
First, we know that $f(s)=\sum_{p} a_p p^{-s}$ is in $\BMOA\cap \mathcal{D}$ if and only if \eqref{eq:bmo} of Corollary~\ref{cor:prime} holds, and by the Cauchy--Schwarz inequality, this implies in particular that the abscissa of absolute convergence is $0$. Hence
\[ \sum_{p} \pm p^{-\alpha-s} \]
can not be in $\BMOA\cap \mathcal{D} $ for any choice of the signs $\pm$ when $1/2<\alpha<1$, although, from an $\ell^2$ point of view, the coefficients decay fast when $\alpha$ is close to $1$. Second, in view of \eqref{eq:sum}, none of the Dirichlet series
 \[f(s):=\sum_{n=2}^\infty \pm \frac{1}{\sqrt {n}} \exp\Big(-c \sqrt{\log n\log\log n}\Big)\, n^{-s},\quad  0<c<1/\sqrt 2, \] 
with random signs $\pm$ can be in $\BMOA\cap \mathcal{D}$, again in spite of fairly good $\ell^2$ decay of the coefficients.

These observations indicate that we should impose an extra condition to obtain a result of the same strength as that of Theorem~\ref{dur}. In fact, they suggest that a possible remedy could be to consider integers generated by a very thin sequence of primes. We will therefore assume that we are in this situation with a fixed set $\mathcal{P}_0$ (finite or not) of prime numbers. We will measure the thinness of this set in terms of its distribution function
 \[ \pi_{0}(x):=\sum_{p\in \mathcal{P}_0,
 p\leq x} 1. \]
 We will say that $\mathcal{P}_0$ is an ultra-thin set of primes if 
 \begin{equation} \label{eq:ultra} 
\int_{3}^{\infty}  \frac{\pi_{0}(x)\log\log x}{x\log^{3}x}dx<\infty ,
 \end{equation}
 and we declare the numbers $w_1=w_2=1$,
 \[ w_n:=\int_{n}^{\infty}  \frac{\pi_{0}(x)\log\log x}{x\log^{3}x}dx , \quad n\ge 3,\]
 to constitute the weight sequence of $\mathcal{P}_0$. We denote by $\mathcal{N}_0$  the set of all  $\mathcal{P}_0$-smooth integers, i.e., the set of positive integers with all their prime divisors belonging to $\mathcal{P}_0$. Our extension of Theorem~\ref{dur} now reads as follows.
 \begin{theorem}\label{proba} Let $\mathcal{P}_0$ be an ultra-thin set of primes with weight sequence $(w_n)$. If
 \begin{equation} \label{eq:durendir} \sum_{n\in \mathcal{N}_0} |a_n|^2 w_n  \log^{2} n <\infty, \end{equation}
 then the Dirichlet series $ f_{\omega}(s)=\sum_{n\in \mathcal{N}_0} \varepsilon_n a_n n^{-s}$ is almost surely in $\BMOA\cap\mathcal{D}$.  
 \end{theorem}
 Let us first note that this is in fact a true extension of Theorem~\ref{dur}, i.e., it reduces to Theorem~\ref{dur} when $\mathcal{P}_0$ consists of a single prime. To see this, we first observe that if 
 $\pi_0(x) \ll \log^{\delta} x$ for some $\delta$, $0\le \delta < 2$, then $\mathcal{P}_0$ is ultra-thin and
 $w_n \ll (\log\log n)/\log^{2-\delta} n$. In particular, in the special case when 
$\mathcal{P}_0$ is a finite set, we find that $w_n\asymp (\log\log n)/\log^2 n$ and hence 
the series in \eqref{eq:durendir} becomes  $\sum_{n\in \mathcal{N}_0} |a_n|^2 \log\log n $. 
If $\mathcal{P}_0$ consists of a single prime $p$, then the Dirichlet series over $\mathcal{N}_0$ 
becomes a Taylor series in the variable $z:=p^{-s}$ and $\log\log n=\log k + \log\log p\sim \log k$ 
for $n=p^k$, and hence \eqref{eq:durendir} becomes the condition of Theorem~\ref{dur}. Finally, 
we note that, plainly, the Dirichlet series over the numbers $p^k$ will be in $\BMOA(\mathbb{C}_0)$ 
if and only if the corresponding Taylor series in the variable $z$ is in $\BMOA(\T)$. 
In view of this relation between Theorem~\ref{dur} and Theorem~\ref{proba}, we see by again appealing to \cite{ACP} that we cannot replace $\log^2 n$ by any sequence growing at a slower rate.
 
For the proof of Theorem~\ref{proba}, we begin by observing that for fixed $\sigma>0$, we have
\[ \mathbb{E}\Big(\int_{-\infty}^\infty \frac{|f_{\omega}(\sigma+it)|^2}{t^2+1}dt\Big)=\pi \sum_{n=1}^\infty |a_n|^2 n^{-2\sigma}\leq \pi \sum_{n=1}^\infty |a_n|^2, \]
and hence $f_{\omega}$ is almost surely in $H_{\operatorname{i}}^{2}(\C_0)$. This means that we may base our proof on Lemma~\ref{basaux}.

The rest of the proof of Theorem~\ref{proba} relies on a lemma from \cite {BCQ} (see also \cite[Theorem 5.3.4]{MAHE}) which is deduced, via the Bohr lift, from a multivariate analogue of a classical  inequality of  Salem and Zygmund due to Kahane \cite[Thm. 3, Sect. 6]{KAH2}.
 \begin{lemma}\label{trois}There exists an absolute constant $C$ such that if  $P(s)=\sum_{k=1}^n a_k k^{-s}$ is a $\mathcal{P}_0$-smooth Dirichlet polynomial of length $n \ge 3$ and $P_\omega$ the corresponding randomized polynomial, then 
 \[ \mathbb{E}(\Vert P_\omega\Vert_\infty) \le C \big(\sum_{k=1}^n |a_k|^2\big)^{1/2}\sqrt{\pi_{0}(n)}\sqrt{\log\log n}.\] 
 \end{lemma}
 Here the price we pay for estimating the uniform norm on the whole of $\mathbb{R}$ is this additional factor $\sqrt{\pi_{0}(n)}$. By considering the randomization (i.e. adding random signs) of the Dirichlet polynomial $\sum_{1\leq k\leq N}p_k^{-s}$ (or randomizing more complicated polynomials of the form  $\sum_{1\leq k\leq N}p_k^{-s}g(p_{N+k}^{-s})$), with a fixed standard polynomial $g$, we see that this extra factor is more or less mandatory.

 \begin{proof}[Proof of Theorem~\ref{proba}] We may for convenience assume that $a_2=0$. Let $X$ be the  random variable defined by 
 \begin{equation}\label{rava}X(\omega):=\int_{0}^1 \sigma\, \Vert T_{\sigma} f'_{\omega} \Vert_{\infty}^{2}d\sigma.\end{equation}
 We will prove that $\mathbb{E}(X)<\infty$. This will imply that $X(\omega)<\infty$ a.s., hence that $f_\omega$ is in $\BMOA\cap \mathcal{D}$ a.s. in view of Lemma \ref{basaux}. 
 
 We fix $\sigma>0$ and set 
 \[ S(x,t):=-\sum_{3\le j\le x} \varepsilon_{j} a_j (\log j) j^{-it} \quad \text{and} \quad  B(x):=\Big(\sum_{3\le j\le x} |a_j|^2 \log^{2}j\Big)^{1/2}. \]
 Since  $(T_{\sigma} f'_{\omega})(it)=-\sum_{n=3}^\infty \varepsilon_n\,a_n (\log n)\, n^{-it} n^{-\sigma}$, we find by partial summation that
 \[ \big|(T_{\sigma} f'_{\omega})(it)\big|\le \int_3^\infty \sigma x^{-\sigma-1}|S(x,t)| dx.\]
 Now using the $L^1-L^2$ Khintchin--Kahane inequality and Lemma~\ref{trois}, we find that
 \begin{equation}\label{khka}\mathbb{E}\big(\big\Vert T_{\sigma} f'_{\omega}\big\Vert_{\infty}^{2}\big)\ll \big(\mathbb{E}\big\Vert T_{\sigma} f'_{\omega}\big\Vert_{\infty}\big)^{2}\ll \Big(\int_{3}^\infty \sigma x^{-\sigma-1} B(x) \sqrt{\pi_{0}(x)} \sqrt{\log\log x} \, dx\Big)^{2} ,\end{equation}
 whence
 \begin{equation} \label{eq:wh}  \mathbb{E}(X) \ll  \int_0^1 \sigma \Big(\int_{3}^\infty \sigma x^{-\sigma-1} B(x) \sqrt{\pi_{0}(x)} \sqrt{\log\log x} \, dx\Big)^{2} d\sigma. \end{equation}
 Setting for convenience $h(x):=B(x) \sqrt{\pi_{0}(x)} \sqrt{\log\log x}$ and using that  for $x,y>1$
 \[ \int_{0}^{1} \sigma^3 (xy)^{-\sigma}d\sigma\leq \int_{0}^{\infty} \sigma^3 (xy)^{-\sigma}d\sigma=\frac{6}{\log^{4}(xy)}, \]
 we find by Fubini's theorem that
 \begin{align*}  \int_0^1 \sigma^3 \Big(\int_{3}^\infty  x^{-\sigma-1} h(x) \, dx\Big)^{2} d\sigma &  \le 6 \int_{3}^{\infty}\int_{3}^{\infty} \frac{h(x) h(y)}{xy \log^4 (xy)} dxdy \\
&\le \frac{3}{4} \int_{3}^{\infty}\int_{3}^{\infty} \frac{h(x) h(y)}{(\log x \log y)^{3/2}} \frac{dxdy}{xy \log (xy)} \le \frac{3 \pi}{4} \int_{3}^{\infty} \frac{h(x)^2}{x\log^3 x} dx.
 \end{align*}
Here we used in the last step that 
 \[ \int_{1}^\infty \int_{1}^\infty\psi(x)\psi(y)\frac{dxdy}{xy (\log xy)}\leq \pi \int_{1}^\infty\psi^{2}(x) \frac{dx}{x}\]
 holds for a nonnegative function $\psi$, which we recognize as   
 Hilbert's inequality \cite[Thm. 316]{HLP}
 \[ \int_{0}^\infty \int_{0}^\infty\varphi(u)\varphi(v)\frac{dudv}{u+v}\leq \pi \int_{0}^\infty \varphi^{2}(u)du\]
for $\varphi(u):=\psi(e^{u})$, after the change of variables $u=\log x$, $v=\log y$.
  
  Hence, returning to \eqref{eq:wh}, we see that
 \begin{equation}\label{eq:change}  \mathbb{E}(X) \ll \int_3^{\infty} \frac{B^2(x) \pi_0(x) \log\log x}{x \log^3 x } dx. \end{equation}
 Now using the definition of $B^2(x)$ as a finite sum and changing the order of integration and summation, we observe that the right-hand side of \eqref{eq:change} equals the series in
 \eqref{eq:durendir}, and hence we conclude that $\mathbb{E}(X)<\infty$.
  \end{proof}


\section{Comparison of norms for Dirichlet polynomials}\label{sec:poly}\label{sec:compare}

We will now establish some relations between the various norms considered so far, when computed for Dirichlet polynomials of fixed length. Throughout this section, our Dirichlet polynomials will be denoted by $f$ and not $P$ as before. Our results complement the main result of \cite{DP} which shows that the supremum of the 
ratio $\| f \|_q/\| f \|_{q'}$ for nonzero Dirichlet polynomials $f$ of length $N$ is
\begin{equation} \label{eq:comp} \exp\left((1+o(1))\frac{\log N}{\log\log N}\log \sqrt{q/q'} \right)   \end{equation}
when $1\le q'<q < \infty$.


We begin with comparisons involving 
$\BMOA$ and $\mathfrak{B}$. For the purpose of  this discussion, it will be convenient to agree that
\[  \| f \|_{\BMOA}^2:= \sup_{h>0} \frac{1}{h} \sup_{t\in \mathbb{R}} \int_{0}^h \int_{t}^{t+h} |f'(\sigma+i \tau)|^2 \sigma d\tau d\sigma,\]
in accordance with the Carleson measure condition of Lemma~\ref{basaux}. We denote by $\mathcal{D}_N$ the space of Dirichlet polynomials of length $N$ vanishing at $+\infty$. The respective ratios $\| f \|_{\infty}/\|f\|_{\mathfrak{B}}$ and $\| f \|_{\BMOA}/\| f\|_{\mathfrak{B}}$ are quite modest compared to \eqref{eq:comp}:

\begin{theorem}\label{unus} 
When $N\to \infty$, we have 
\begin{align} \label{eq:asymp} \sup_{f\in \mathcal{D}_N\setminus \{0\}} \frac{\| f \|_{\infty}}{\| f \|_{\mathfrak{B}}} & \sim \log\log N, \\ \label{eq:asymp1}
\sup_{f\in \mathcal{D}_N\setminus \{0\}} \frac{\| f \|_{\BMOA}}{\| f \|_{\mathfrak{B}}} & \asymp \sqrt{\log\log N},\\
\label{eq:asymp2}\sup_{f\in \mathcal{D}_N\setminus\{0\}} \frac{\Vert f\Vert_\infty}{\Vert f\Vert_{BMOA}} & \asymp \log \log N .
\end{align}
\end{theorem}
We require two new lemmas. The first contains two versions of Bernstein's inequality.
\begin{lemma}[Bernstein inequalities] \label{es}  We have 
\begin{equation} \label{eq:bern} \| f' \|_\infty \le \log N \| f \|_{\infty} \quad \text{and} \quad \Vert f'\Vert_\infty \leq 4\log N \Vert f\Vert_{\mathfrak{B}}\end{equation}
for every $f$ in $\mathcal{D}_N$.  
\end{lemma}
The first inequality in \eqref{eq:bern} is a special case of a general version of Bernstein's inequality for finite sums of purely imaginary exponentials (see \cite[p. 30]{KAH}). We will find that the second inequality is a consequence of the next lemma.
\begin{lemma}\label{mardi} We have  
\[ \| f \|_{\infty} \le \frac{1}{(1-c)} \| T_{c/\log N}f \|_{\infty} \]
for every Dirichlet polynomial $f$ in $\mathcal{D}_N$, when $0<c<1$ and $N\ge 2$.
\end{lemma}
\begin{proof}
The first inequality in \eqref{eq:bern} and the maximum modulus principle give for any fixed $\sigma>0$
\[ |f(it)-f(\sigma+it)|\leq \sigma \| f'\Vert_{\infty}\leq \sigma \log N\Vert f\Vert_{\infty}. \] 
Hence, setting $\sigma=c/\log N$, we see that
\[ |f(it)|\leq \big|\big(T_{c/\log N}f\big)(it)\big|+c  \| f \|_{\infty} \]
from which the result follows. 
\end{proof}

\begin{proof}[Proof of the second inequality in \eqref{eq:bern}]
Using the definition of the Bloch norm, we see for any fixed $\sigma>0$ that 
\[ \| f\|_{\mathfrak{B}}\geq \sup_{t\in \mathbb{R}} \sigma |f'(\sigma+it)|. \]
Setting $\sigma=c/\log N$ and applying Lemma~\ref{mardi} to $f'$, we then get
\[ \| f\|_{\mathfrak{B}} \ge \frac{c(1-c)}{\log N} \| f' \|_{\infty} .\] Choosing $c=1/2$, we obtain the asserted result.
\end{proof}

\begin{proof}[Proof of Theorem~\ref{unus}]
Combining \eqref{eq:point} and Lemma~\ref{mardi}, we find that if $f(+\infty)=0$, then
\begin{equation} \label{eq:bloch} \|f\|_{\infty} \le \frac{\log\log N + \log (1/c) + C}{(1-c)} \| f \|_{\mathfrak{B}}. \end{equation}
Choosing $c=1/\log\log N$, we obtain
\[ \frac{\|f\|_{\infty}}{\| f\|_{\mathfrak{B}}}\le \log\log N +O(\log\log\log N), \]
assuming that $f\neq 0$.  On the other hand, the polynomial $f(s)=\sum_{n=2}^N \frac{1}{n\log n} n^{-s}$ satisfies $\| f \|_{\infty} = \log\log N + O(1)$, while 
\[ |f'(s)| \le \sum_{n=2}^{\infty} n^{-\sigma-1}\leq \zeta(\sigma+1)-1\leq \frac{1}{\sigma}, \]
so that $\| f \|_{\mathfrak{B}} \le 1$. Hence we have shown that the supremum over $f$ of  the left-hand side of \eqref{eq:bloch} exceeds $\log\log N+O(1)$. We conclude that \eqref{eq:asymp} holds.


We now use Lemma \ref{basaux} to estimate $\| f \|_{\BMOA}$ under the assumption that $f$ is in $\mathcal{D}_N$ and $\| f \|_{\mathfrak{B}}=1$. We first observe that if $h\le 1/\log N$, then by the second Bernstein inequality of Lemma~\ref{es},
\[ \int_{0}^h \int_{t}^{t+h} |f'(\sigma+i \tau)|^2 \sigma d\tau d\sigma \le 16 (\log N)^2 h \int_{0}^h \sigma d\sigma \le 8h.\]
On the other hand, if $1/\log N<h \le 1$, then we obtain by the same argument
\[ \int_{0}^h \int_{t}^{t+h} |f'(\sigma+i \tau)|^2 \sigma d\tau d\sigma \le 8h + \int_{1/\log N}^h \int_{t}^{t+h} |f'(\sigma+i \tau)|^2 \sigma d\tau d\sigma .\]
Using the bound $|f'(\sigma+i \tau)| \le 1/\sigma$ in the integral term, where $1/\log N \le \sigma \le h\le1$,  we infer from this that
\[ \| f \|_{\BMOA}^2 \le  \log\log N+O(1). \] 
The optimality of the latter bound is seen by considering  the function
\[ g(s):=\sum_{k\le \log\log N}  \left[e^{e^k}\right]^{-s} , \]
that satisfies $\| g \|_{\mathfrak{B}} \asymp 1$ and $\| g \|_{\BMOA}^2 \asymp \log \log N$. Here the first relation is trivial, and the second follows from \eqref{eq:feff} of Theorem~\ref{thm:fefferman}. Hence \eqref{eq:asymp1} has been established.

Finally, to prove  \eqref{eq:asymp2}, we first infer from \eqref{eq:asymp} that
\[ \Vert f\Vert_\infty\ll  \log\log N\Vert f\Vert_{\mathcal{B}}\ll  \log\log N \Vert f\Vert_{BMOA}. \]
The example 
$f(s)=\sum_{2\leq n\leq N} \frac{1}{n\log n} n^{-s}$ used above satisfies $\Vert f\Vert_{BMOA}\asymp 1$ by \eqref{eq:hilbert}  and trivially $\Vert f\Vert_{\infty}\asymp \log\log N$. This establishes the reverse inequality in \eqref{eq:asymp2}.
\end{proof}

We close this section by establishing a lemma that will be used in two different ways in the next section. In contrast to the preceding comparison results, as well as those of \cite{DP}, Lemma~\ref{lem:finite} is a purely multiplicative result, and we therefore state it for polynomials in several complex variables.

\begin{lemma}\label{lem:finite}
There exists an absolute constant $C$ such that if $F$ is a holomorphic polynomial of degree $d\ge 2$ in $n\ge 1$ complex variables, then
\begin{equation} \label{eq:sub} \| F \|_{\infty} \le C \| F\|_{n \log d}. \end{equation}
\end{lemma}
\begin{proof}    
We now apply a multi-dimensional version of Bernstein's inequality, namely 
\[ |F(z)-F(w)|\leq \frac{\pi}{2}  d \Vert z-w\Vert_\infty \Vert F\Vert_\infty ,\]
which holds for holomorphic polynomials $F$ in $n$ complex variables and all points $z=(z_j)$  and  $ w=(w_j)$ on  $\T^n$  (see \cite[pp. 125--126]{MAHE}). This implies that if $w$ is a point on $\T^n$ at which $|F(w)|=\| F\|_\infty $, then  $|F(z)|\ge \| F \|_\infty/2$ whenever we have
$|w_j-z_j|\le \frac{c}{d} $ for $j\le \pi_0(n)$ with  $c:=1/\pi$. It follows that 
\[ \| F \|_q\ge \frac{1}{2} (2c)^{n/q}  d^{-n/q} \| F\|_\infty  \] 
and hence we get
\[  \| F\|_\infty \le 2e (2c)^{-1/\log d}  \| F\|_{n \log d} \le 2 \pi^{1/\log 2}  \| F\|_{n \log d}. \] 
\end{proof}

\section{The partial sum operator for Dirichlet series and Riesz projection on $\T$}\label{sec:partial}

We will now make some remarks about the partial sum operator $S_N$ which is defined by the formula
\[ (S_N f)(s):=\sum_{n\le N} a_n n^{-s} \]
for $f(s)=\sum_{n=1}^{\infty} a_n n^{-s}$. We are interested in computing the norm of $S_N$ when it acts on $\mathcal{H}^q$. In what follows, we denote this norm by $\| S_N \|_q$. Most of what is known about $\|S_N\|_q$ for different values of $q$ and $N$ can be deduced from an idea that goes back to Helson \cite{He}, by which we may effectively rewrite $S_N$ as a one-dimensional Riesz projection. We will now state and prove a theorem in this vein that can be obtained almost immediately by combining \cite[Thm. 8.7.2]{R} with the optimal bounds of Hollenbeck and Verbitsky \cite{HV} for Riesz projection on $\T$. We choose to offer a detailed proof, however, because it makes the transference to one-dimensional Riesz projection explicit and leads to nontrivial quantitative estimates.

We will consider a somewhat more general situation to emphasize the main idea of the transference to the unit circle. To this end, we fix a completely multiplicative function $g(n)\ge 1$ such that $g(n)\to \infty$ when $n\to \infty$. By considering $g(p^k)$ for $k\geq 1$, we see that this means that
$g(p)>1$  for all primes $p$ and that $\lim_{p\to\infty} g(p)=\infty.$ We then introduce the projection
\[ P_{g,x} \left(\sum_{n=1}^\infty a_n n^{-s}\right) := \sum_{g(n)\le x} a_n n^{-s}. \]
We see that $S_N=P_{g,N}$ in the special case when $g(n)=n$. 

 
\begin{theorem}\label{show}  Suppose that $g$ is a completely multiplicative function taking only positive values and  that $g(n)\to\infty$ when $n\to\infty$. Then
\begin{equation}\label{key}   \sup_{x\ge 1} \| P_{g,x} \|_{\mathcal{H}^q}= \frac{1}{\sin(\pi/q)} \end{equation}
for $1<q<\infty$. 
\end{theorem} 
\begin{proof} 
We consider first the easy direction, namely that $\sup_{x\geq 1} \| P_{g,x} \|_q\geq \frac{1}{\sin(\pi/q)}$. It is classical and straightforward to check that the norm of the Riesz projection equals $\sup_{N\geq 1} \| \widetilde S_N \|_q,$ where $\widetilde S_N$ is the 1-dimensional partial sum operator acting on $H^q(\T)$. On the other hand, clearly $\| P_{g,g(2^N)} \|_q\ge \| \widetilde S_{N} \|_q$, so the claim follows from the fact that the bound of Hollenbeck and Verbitsky  is optimal.

In order to treat the more interesting direction, we begin by fixing a positive integer $Q$ that will be specified later, depending on $x$. Then for every prime $p$, we choose a positive integer $m_p$ such that 
\[ \left| Q \log g(p) -m_p\right| \le \frac{1}{2}. \]
This is possible because $g(p)>1$ by the assumption that $g(n)\to \infty$. Now let $z$ be a point on the unit circle. Write $n$ in multi-index notation as $n=p^{\alpha(n)}=\prod_{p} p^{\alpha_p(n)}$, set accordingly $\beta(n)=\sum_{p} \alpha_{p}(n) m_p$ and consider the transformation
\[ T_{g,Q,z} \left(\sum_{n=1}^\infty a_n n^{-s}\right)=\sum_{n=1}^\infty a_n z^{\beta(n)} n^{-s}. \]
Taking the Bohr lift, we see that the effect of $T_{g,Q,z}$ acting on $f$ is that each variable is multiplied by a unimodular number. This shows that $T_{g,Q,z}$ acts isometrically on $\mathscr{H}^q$ for every $q>0$.

Note that by construction
\[ \left| \beta(n) -  Q \log g(n) \right|\le \frac{1}{2}\big|\alpha(n)\big|=\frac{1}{2}\Omega(n), \]
where $\Omega(n)$ is the number of prime factors of $n$ counting with multiplicity.
We now choose the parameter $Q$ so large that 
\begin{equation}\label{eq:req} 
\max_{g(n)\leq x }\beta(n)< \inf_{g(n)>x}\beta(n). 
\end{equation} 
This is obtained if
\begin{equation} \label{eq:sep} \inf_{g(n)>x} \big(Q\log g(n)- \frac{1}{2} \Omega(n) \big) > \max_{g(n)\le x} \big(Q \log g(n) + \frac{1}{2} \Omega(n) \big). \end{equation}
We may achieve \eqref{eq:sep} because the assumptions on $g$ 
 imply that $\log g(n) \ge c \Omega(n) $ for some $c>0$. Namely, this inequality clearly yields that
 \[ \inf_{g(n)>x} \big(Q\log g(n)- \frac{1}{2} \Omega(n) \big)\ge (Q-c^{-1}/2) \log (x+1) \]
for some $\varepsilon>0$, while on the other hand
\[    \max_{g(n)\le x} \big(Q \log g(n) + \frac{1}{2} \Omega(n) \big) \le (Q+c^{-1}/2) \log x.\]
Having made this choice of $Q$, we see that \eqref{eq:req} ensures that we may write 
\[ (T_{g,Q,z} P_{g,x} f)(s)=\sum_{\beta(n)\le x'} a_n z^{\beta(n)} n^{-s}\]
for a suitable $x'$.
Hence,
using the Bohr lift $B$,  the translation invariance of  $m_\infty$ under $T_z$ with $ T_{z}(w)=(w_p z^{m_p})$, Fubini's theorem, and Hollenbeck and Verbitsky's theorem \cite{HV} on the $L^q$ norm of the Riesz projection on $\T$, we get successively:
\begin{align*}
\Vert S_{N}(f)\Vert_{q}^{q}& =\int_{\T^\infty} \big|S_{N}(Bf)(w)\big|^{q}dm_{\infty}(w)\\
&=\int_{\T}\Big(\int_{\T^\infty} \big|S_{N}(Bf)(T_{z} w)\big|^{q}dm_{\infty}(w)\Big)dm(z)\\
&=\int_{\ T^\infty}\Big(\int_{\T} \big|\sum_{n=1}^N a_n w^{\alpha(n)}z^{\beta(n)}
\big|^{q}dm(z)\Big)dm_{\infty}(w)\\
&\leq \Big(\frac{1}{\sin \pi/q}\Big)^q \int_{\T^\infty}\Big(\int_{\T} \big|\sum_{n=1}^\infty a_n w^{\alpha(n)}z^{\beta(n)}\big|^{q}dm(z)\Big)dm_{\infty}(w)\\
&= \Big(\frac{1}{\sin \pi/q}\Big)^q \Vert f\Vert_{q}^{q}.\end{align*}
\end{proof}

If we specialize to the case when $g(n)=n$  and $x= N$, it is of interest to see how large the intermediate parameter $Q$ has to be to ensure that \eqref{eq:sep} holds.
We see that 
this  happens if
\begin{equation} \label{eq:special} Q\log(N+j)-\frac{\log (N+j)}{2\log 2}>Q\log N+\frac{\log N}{2\log 2}.\end{equation}
when $j=1,2 ...$. We may assume that $Q>1/(2\log 2)$ so that
\[ Q\log(N+j)-\frac{\log (N+j)}{2\log 2} \ge Q \log N - \frac{\log N}{2\log 2} + \Big(Q- \frac{1}{2\log 2}\Big) \frac{1}{2N}. \]
This shows that \eqref{eq:special} holds if we choose 
\begin{equation} \label{eq:require} Q\ge c N \log N \end{equation}
   with $c>0$ large enough. Since $T_{n,Q} S_N f$ will be a polynomial of degree at most $Q \log N+(\log N)(2\log 2)$ in the dummy variable $z$, we may now, following again the reasoning of the above proof, use Lemma~\ref{lem:finite} with $n=1$ and $d= Q \log N+O(\log N)$ to deduce that $\| S_N \|_{\infty} \ll \log N $. We thus recapture a result that was first established in \cite{BCQ} by use of Perron's formula and contour integration.

The bound just obtained remains the best known upper bound for $\| S_N\|_{\infty} $. On the other hand, it is known that $\| S_N \|_{\infty}\gg \log \log N$ (obtained for Dirichlet series over powers of a single prime). We are thus far from knowing the right order of magnitude of $\| S_N\|_{\infty}$. A similar situation persists when $q=1$ in which case we have $\log\log N \ll \| S_N\|_1 \ll \log N/\log\log N$ by a result of \cite{BBSS}.

We will now show that if we specialize to Dirichlet series over $\mathcal{P}_0$-smooth numbers, then the estimate in the case $q=\infty$ can be improved for certain ultra-thin sets of primes $\mathcal{P}_0$. To this end, we denote by $\mathcal{H}^q(\mathcal{P}_0)$ the subspace of $\mathcal{H}^{q}$ consisting of Dirichlet series over the sequence $\mathcal{N}_0$ of $\mathcal{P}_0$-smooth numbers, and we let $\| S_{N}\|_{\mathcal{H}^q(\mathcal{P})_0}$ be the norm of $S_N$ when restricted to $\mathcal{H}^q(\mathcal{P}_0)$. 

The crucial observation is that it may now be profitable to apply Lemma~\ref{lem:finite} \emph{before} we make the transference to one-dimensional Riesz projection. Indeed, we observe that the Bohr lift of a Dirichlet polynomial of length $N$ over $\mathcal{P}_0$-smooth numbers will be a polynomial of degree at most $\log N/\log 2$ in $\pi_0(N)$ complex variables. Hence the norm on the right-hand side of \eqref{eq:sub} can be taken to be the $\pi_0(N) \log \log N$-norm. Combining this observation with Theorem~\ref{show}, we then get the following result which yields an improvement when $\pi_0(x)=o(\log x/\log\log x)$.

\begin{theorem}\label{thm:reis} There exists an absolute constant $C$ such that 
 \begin{equation}\label{d}\Vert S_N\Vert_{\mathcal{H}^{\infty}(\mathcal{P}_0)} \leq C \pi_0(N) \log\log N  \end{equation}
 when $\pi_0(N)\ge 1$ and $\log\log N\ge 2$.
\end{theorem}
Following the proof of \cite[Thm. 5.2]{BBSS} word for word, we may obtain a similar result for $\| S_N \|_{\mathcal{H}^{1}(\mathcal{P}_0)}$ with $\pi_0(N)\log\log N$ replaced by the logarithm of the maximal order of the divisor function at $N$ when restricted to $\mathcal{N}_0$. In contrast to \eqref{d}, this bound is nontrivial for all sets of primes $\mathcal{P}_0$. In particular, it yields $\| S_N\|_1 \ll \log\log N$ when $\mathcal{P}_0$ is a finite set and $\| S_N\|_1\ll \log N/\log\log N$ when $\mathcal{P}_0$ is the set of all primes, since then the logarithm of the maximal order of the divisor function at $N$ is $O(\log N/\log\log N)$.


\end{document}